\documentclass[12pt,twoside]{article}

\usepackage{amsmath}
\usepackage{latexsym}
\usepackage{amssymb}
\usepackage[dvips]{graphicx}

\usepackage{amssymb,amsthm}

\newtheorem{theorem}{Theorem}
\newtheorem{proposition}[theorem]{Proposition}
\newtheorem{lemma}[theorem]{Lemma}
\newtheorem{corollary}[theorem]{Corollary}

\newtheorem{definition}[theorem]{Definition}
\newtheorem{remark}{Remark}

\date{}

\begin{document}

\emph{{\footnotesize Submitted 24-Feb-2010; accepted 12-Oct-2010; to \emph{Int. J. Open Problems Comp. Math.}}}

\emph{}

\emph{}

\centerline{}

\centerline{}

\centerline{}
\centerline{}
\centerline {\Large{\bf Delta-Nabla Isoperimetric Problems}}

\centerline{}

%-------------------------------------------------------

\newcommand{\mvec}[1]{\mbox{\bfseries\itshape #1}}

\centerline{\bf {Agnieszka B. Malinowska, Delfim F. M. Torres}}

\centerline{}

\centerline{Bia{\l}ystok University of Technology,
15-351 Bia{\l}ystok, Poland}
\centerline{e-mail: abmalinowska@ua.pt}

\centerline{University of Aveiro,
3810-193 Aveiro, Portugal}
\centerline{e-mail: delfim@ua.pt}

%-------------------------------------------------------

\newtheorem{Theorem}{\quad Theorem}[section]

\newtheorem{Definition}[Theorem]{\quad Definition}

\newtheorem{Corollary}[Theorem]{\quad Corollary}

\newtheorem{Lemma}[Theorem]{\quad Lemma}

\newtheorem{Example}[Theorem]{\quad Example}

\begin{abstract}
\textbf{\emph{We prove general necessary optimality conditions
for delta-nabla isoperimetric problems of the calculus of variations.}}
\end{abstract}


\medskip


%=============================
\section{Introduction}
%=============================

Isoperimetric problems consist in maximizing or minimizing
a cost functional subject to integral constraints.
They have found a broad class of
important applications throughout the centuries.
Areas of application include astronomy, geometry,
algebra, and analysis \cite{Viktor}.
The study of isoperimetric problems
is nowadays done, in an elegant and rigorously way,
by means of the theory of the calculus of variations
\cite{Brunt}, and concrete isoperimetric problems
in engineering have been investigated
by a number of authors \cite{Curtis}.
For recent developments on isoperimetric problems
we refer the reader to \cite{A:T,A:T:JMAA,F:T:09}
and references therein.

A new delta-nabla calculus of variations
has recently been introduced by the authors
in \cite{Bedlewo:2009}. The new calculus of variations
allow us to unify and extend the two standard approaches
of the calculus of variations on time scales \cite{F:T:08,NM:T,rachid},
and is motivated by applications in economics \cite{Pedregal}.

The delta-nabla variational theory is still in the very beginning,
and much remains to be done. In this note we develop further the theory
by introducing the isoperimetric problem in the delta-nabla setting and proving
respective necessary optimality conditions. Section~\ref{sec:prm} reviews
the Euler-Lagrange equations of the delta-nabla
calculus of variations \cite{Bedlewo:2009} and recalls
the results of the literature needed in the sequel.
Our contribution is given in Section~\ref{sec:mr:iso},
where the delta-nabla isoperimetric problem
is formulated and necessary optimality
conditions for both normal and abnormal extremizers are proved
(see Theorems~\ref{thm:mr:iso} and \ref{th:iso:abn}).
We proceed with Section~\ref{sec:ex},
illustrating the applicability of our results
with an example. Finally, we present the
conclusion (Section~\ref{sec:conc})
and some open problems (Section~\ref{sec:open}).

%=============================
\section{Preliminaries}
\label{sec:prm}
%=============================

We assume the reader to be familiar with the theory of time scales.
For an introduction to the calculus on time scales we refer
to the books \cite{B:P:01,B:P:03,Lak:book}.

Let $\mathbb{T}$ be a given time scale with jump operators
$\sigma$ and $\rho$, and differential operators $\Delta$ and $\nabla$.
Let $a, b \in \mathbb{T}$, $a < b$,
and $\left(\mathbb{T} \setminus \{a,b\}\right)\cap [a,b] \ne \emptyset$;
and $L_{\Delta}(\cdot,\cdot,\cdot)$ and $L_{\nabla}(\cdot,\cdot,\cdot)$ be two given smooth
functions from $\mathbb{T} \times \mathbb{R}^2$ to $\mathbb{R}$.
The results here discussed are trivially generalized for
admissible functions $y : \mathbb{T}\rightarrow\mathbb{R}^n$
but for simplicity of presentation we restrict ourselves to the scalar case $n=1$.
Throughout the text we use the operators $[y]$ and $\{y\}$ defined by
\begin{equation*}
[y](t) := \left(t,y^\sigma(t),y^\Delta(t)\right) \, ,
\quad \{y\}(t) := \left(t,y^\rho(t),y^\nabla(t)\right).
\end{equation*}

In \cite{Bedlewo:2009} the problem of extremizing
a delta-nabla variational functional
subject to given boundary conditions
$y(a) = \alpha$ and $y(b) = \beta$ is posed and studied:
\begin{equation}
\label{problem:P}
\begin{gathered}
\mathcal{J}(y) =
\left(\int_a^b L_{\Delta}[y](t) \Delta t\right)
\left(\int_a^b L_{\nabla}\{y\}(t) \nabla t\right) \longrightarrow
\textrm{extr} \\
y \in C_{\diamond}^1\left([a,b],\mathbb{R}\right) \\
y(a) = \alpha \, , \quad y(b) = \beta \, ,
\end{gathered}
\end{equation}
where $C_{\diamond}^1\left([a,b],\mathbb{R}\right)$
denote the class of functions
$y : [a,b]\rightarrow\mathbb{R}$  with
$y^\Delta$ continuous on $[a,b]^\kappa$
and $y^\nabla$ continuous on $[a,b]_\kappa$.

\begin{definition}
We say that $\hat{y}\in C_{\diamond}^{1}([a,b],
\mathbb{R})$ is a weak local minimizer (respectively weak local
maximizer) for problem \eqref{problem:P} if there exists
$\delta >0$ such that
$\mathcal{J}(\hat{y})\leq \mathcal{J}(y)$
(respectively $\mathcal{J}(\hat{y})
\geq \mathcal{J}(y)$)
for all $y \in C_{\diamond}^{1}([a,b], \mathbb{R})$
satisfying the boundary conditions
$y(a) = \alpha$ and $y(b) = \beta$, and
$||y - \hat{y}||_{1,\infty} < \delta$,
where
$||y||_{1,\infty}:=
||y^{\sigma}||_{\infty}
+ ||y^{\rho}||_{\infty}
+ ||y^{\Delta}||_{\infty}
+ ||y^{\nabla}||_{\infty}$
and
$||y||_{\infty} :=\sup_{t \in [a,b]_{\kappa}^{\kappa}}|y(t)|$.
\end{definition}

The main result of \cite{Bedlewo:2009} gives two different forms
for the Euler--Lagrange equation on time scales
associated with the variational problem \eqref{problem:P}.

\begin{theorem}[The general Euler-Lagrange equations on time scales \cite{Bedlewo:2009}]
\label{thm:mr}
If $\hat{y} \in C_{\diamond}^1$ is a weak local extremizer of problem
\eqref{problem:P}, then $\hat{y}$ satisfies
the following delta-nabla integral equations:
\begin{multline}
\label{eq:EL1}
\mathcal{J}_\nabla(\hat{y})
\left(\partial_3 L_\Delta[\hat{y}](\rho(t))
-\int_{a}^{\rho(t)} \partial_2 L_\Delta[\hat{y}](\tau) \Delta\tau\right)\\
+ \mathcal{J}_\Delta(\hat{y})
\left(\partial_3 L_\nabla\{\hat{y}\}(t)
-\int_{a}^{t} \partial_2 L_\nabla\{\hat{y}\}(\tau) \nabla\tau\right)
= \text{const} \quad \forall t \in [a,b]_\kappa \, ;
\end{multline}
\begin{multline}
\label{eq:EL2}
\mathcal{J}_\nabla(\hat{y})
\left(\partial_3 L_\Delta[\hat{y}](t)
-\int_{a}^{t} \partial_2 L_\Delta[\hat{y}](\tau) \Delta\tau\right)\\
+ \mathcal{J}_\Delta(\hat{y})
\left(\partial_3 L_\nabla\{\hat{y}\}(\sigma(t))
-\int_{a}^{\sigma(t)} \partial_2 L_\nabla\{\hat{y}\}(\tau) \nabla\tau\right)
= \text{const} \quad \forall t \in [a,b]^\kappa \, .
\end{multline}
\end{theorem}

\begin{remark}
In the classical context (\textrm{i.e.},
when $\mathbb{T} = \mathbb{R}$) the
necessary conditions \eqref{eq:EL1}
and \eqref{eq:EL2} coincide with the Euler--Lagrange
equations recently obtained in \cite{Pedregal}.
\end{remark}

Our main goal is to generalize Theorem~\ref{thm:mr}
by covering variational problems subject to isoperimetric constraints.
In order to do it (\textrm{cf.} proof of Theorem~\ref{thm:mr:iso})
we use some relationships of \cite{A:G:02} between the delta and nabla derivatives,
and some relationships of \cite{G:G:S:05} between the delta and nabla integrals.

\begin{proposition}[Theorems~2.5 and 2.6 of \cite{A:G:02}]
\label{prop:rel:der}
(i) If $f : \mathbb{T} \rightarrow \mathbb{R}$ is delta differentiable
on $\mathbb{T}^\kappa$ and $f^\Delta$ is continuous on $\mathbb{T}^\kappa$,
then $f$ is nabla differentiable on $\mathbb{T}_\kappa$ and
\begin{equation}
\label{eq:chgN_to_D}
f^\nabla(t)=\left(f^\Delta\right)^\rho(t) \quad \text{for all }
t \in \mathbb{T}_\kappa \, .
\end{equation}
(ii) If $f : \mathbb{T} \rightarrow \mathbb{R}$ is nabla differentiable
on $\mathbb{T}_\kappa$ and $f^\nabla$ is continuous on $\mathbb{T}_\kappa$,
then $f$ is delta differentiable on $\mathbb{T}^\kappa$ and
\begin{equation}
\label{eq:chgD_to_N}
f^\Delta(t)=\left(f^\nabla\right)^\sigma(t) \quad \text{for all }
t \in \mathbb{T}^\kappa \, .
\end{equation}
\end{proposition}

\begin{proposition}[Proposition~7 of \cite{G:G:S:05}]
If function $f : \mathbb{T} \rightarrow \mathbb{R}$
is continuous, then for all $a, b \in \mathbb{T}$
with $a < b$ we have
\begin{gather}
\int_a^b f(t) \Delta t = \int_a^b f^\rho(t) \nabla t \, , \label{eq:DtoN}\\
\int_a^b f(t) \nabla t = \int_a^b f^\sigma(t) \Delta t \, . \label{eq:NtoD}
\end{gather}
\end{proposition}

We also use the nabla Dubois--Reymond lemma of \cite{NM:T}.

\begin{lemma}[Lemma~14 of \cite{NM:T}]
\label{DBRL:n}
Let $f \in C_{\textrm{ld}}([a,b], \mathbb{R})$. If
$$
\int_{a}^{b} f(t)\eta^{\nabla}(t)\nabla t=0 \quad
\mbox{for all $\eta \in C_{\textrm{ld}}^1([a,b],
\mathbb{R})$ with $\eta(a)=\eta(b)=0$} \, ,
$$
then $f(t)=c$ on $t\in [a,b]_\kappa$
for some constant $c$.
\end{lemma}

%=============================
\section{Main Results}
\label{sec:mr:iso}
%=============================

We consider delta-nabla isoperimetric problems on time scales.
The problem consists of extremizing
\begin{equation}
\label{problem:P:iso}
\mathcal{L}(y) = \left(\int_a^b L_{\Delta}[y](t) \Delta t\right)
\left(\int_a^b L_{\nabla}\{y\}(t) \nabla t\right) \longrightarrow
\textrm{extr}
\end{equation}
in the class of functions $y \in C_{\diamond}^{1}([a,b], \mathbb{R})$
satisfying the boundary conditions
\begin{equation}\label{bou:con}
y(a) = \alpha \, , \quad y(b) = \beta \, ,
\end{equation}
and the constraint
\begin{equation}\label{const}
\mathcal{K}(y) = \left(\int_a^b K_{\Delta}[y](t) \Delta t\right)
\left(\int_a^b K_{\nabla}\{y\}(t) \nabla t\right)=k,
\end{equation}
where $\alpha$, $\beta$, $k$ are given real numbers.

\begin{definition}
We say that $\hat{y}\in C_{\diamond}^{1}([a,b],\mathbb{R})$
is a weak local minimizer (respectively weak local
maximizer) for \eqref{problem:P:iso}--\eqref{const}
if there exists $\delta>0$ such that
$$
\mathcal{L}(\hat{y})\leq \mathcal{L}(y) \quad (\text{respectively} \
\   \mathcal{L}(\hat{y}) \geq \mathcal{L}(y))
$$
for all $y \in C_{\diamond}^{1}([a,b], \mathbb{R})$ satisfying the boundary
conditions \eqref{bou:con}, the isoperimetric constraint
\eqref{const}, and
$||y - \hat{y}||_{1,\infty} < \delta$.
\end{definition}

\begin{definition}
We say that $\hat{y} \in C_{\diamond}^1$ is an extremal for $\mathcal{K}$
if $\hat{y}$ satisfies the delta-nabla integral
equations \eqref{eq:EL1} and \eqref{eq:EL2} for $\mathcal{K}$, \textrm{i.e.},
\begin{multline}
\label{eq:EL1:iso}
\mathcal{K}_\nabla(\hat{y}) \left(\partial_3
K_\Delta[\hat{y}](\rho(t))
-\int_{a}^{\rho(t)} \partial_2 K_\Delta[\hat{y}](\tau) \Delta\tau\right)\\
+ \mathcal{K}_\Delta(\hat{y}) \left(\partial_3
K_\nabla\{\hat{y}\}(t) -\int_{a}^{t} \partial_2
K_\nabla\{\hat{y}\}(\tau) \nabla\tau\right) = \text{const} \quad
\forall t \in [a,b]_\kappa \, ;
\end{multline}
\begin{multline}
\label{eq:EL2:iso}
\mathcal{K}_\nabla(\hat{y}) \left(\partial_3
K_\Delta[\hat{y}](t)
-\int_{a}^{t} \partial_2 K_\Delta[\hat{y}](\tau) \Delta\tau\right)\\
+ \mathcal{K}_\Delta(\hat{y}) \left(\partial_3
K_\nabla\{\hat{y}\}(\sigma(t)) -\int_{a}^{\sigma(t)} \partial_2
K_\nabla\{\hat{y}\}(\tau) \nabla\tau\right) = \text{const} \quad
\forall t \in [a,b]^\kappa \, .
\end{multline}
An extremizer (\textrm{i.e.}, a weak local minimizer or a weak local
maximizer) for the problem \eqref{problem:P:iso}--\eqref{const} that is
not an extremal for $\mathcal{K}$ is said to be a normal extremizer;
otherwise (\textrm{i.e.}, if it is an extremal for $\mathcal{K}$), the
extremizer is said to be abnormal.
\end{definition}

\begin{theorem}
\label{thm:mr:iso}
If $\hat{y} \in C_{\diamond}^1\left([a,b],\mathbb{R}\right)$
is a normal extremizer for the isoperimetric problem
\eqref{problem:P:iso}--\eqref{const}, then there exists $\lambda \in \mathbb{R}$
such that $\hat{y}$ satisfies the following delta-nabla integral equations:
\begin{multline}
\label{iso:EL1}
\mathcal{L}_\nabla(\hat{y}) \left(\partial_3
L_\Delta[\hat{y}](\rho(t))
-\int_{a}^{\rho(t)} \partial_2 L_\Delta[\hat{y}](\tau) \Delta\tau\right)\\
+ \mathcal{L}_\Delta(\hat{y}) \left(\partial_3
L_\nabla\{\hat{y}\}(t) -\int_{a}^{t} \partial_2
L_\nabla\{\hat{y}\}(\tau) \nabla\tau\right) \\
-\lambda\left\{\mathcal{K}_\nabla(\hat{y}) \left(\partial_3
K_\Delta[\hat{y}](\rho(t))
-\int_{a}^{\rho(t)} \partial_2 K_\Delta[\hat{y}](\tau) \Delta\tau\right)\right.\\
\left.+ \mathcal{K}_\Delta(\hat{y}) \left(\partial_3
K_\nabla\{\hat{y}\}(t) -\int_{a}^{t} \partial_2
K_\nabla\{\hat{y}\}(\tau) \nabla\tau\right)\right\} = \text{const}
\quad \forall t \in [a,b]_\kappa \, ;
\end{multline}
\begin{multline}
\label{iso:EL2}
\mathcal{L}_\nabla(\hat{y}) \left(\partial_3
L_\Delta[\hat{y}](t)
-\int_{a}^{t} \partial_2 L_\Delta[\hat{y}](\tau) \Delta\tau\right)\\
+ \mathcal{L}_\Delta(\hat{y}) \left(\partial_3
L_\nabla\{\hat{y}\}(\sigma(t)) -\int_{a}^{\sigma(t)} \partial_2
L_\nabla\{\hat{y}\}(\tau) \nabla\tau\right) \\
-\lambda \left\{\mathcal{K}_\nabla(\hat{y}) \left(\partial_3
K_\Delta[\hat{y}](t)
-\int_{a}^{t} \partial_2 K_\Delta[\hat{y}](\tau) \Delta\tau\right)\right.\\
+ \left.\mathcal{K}_\Delta(\hat{y}) \left(\partial_3
K_\nabla\{\hat{y}\}(\sigma(t)) -\int_{a}^{\sigma(t)} \partial_2
K_\nabla\{\hat{y}\}(\tau) \nabla\tau\right)\right\}= \text{const}
\quad \forall t \in [a,b]^\kappa \, .
\end{multline}
\end{theorem}

\begin{proof}
Consider a variation of $\hat{y}$, say $\bar{y}=\hat{y} +
\varepsilon_{1} \eta_{1}+\varepsilon_{2} \eta_{2}$, where for each
$i\in \{1,2\}$, $\eta_{i}\in C_{\diamond}^{1}([a,b],\mathbb{R})$
and $\eta_{i}(a)=\eta_{i}(b)=0$, and $\varepsilon_{i}$ is a
sufficiently small parameter ($\varepsilon_{1}$ and
$\varepsilon_{2}$ must be such that
$||\bar{y}-\hat{y}||_{1,\infty}<\delta$ for some $\delta>0$).
Here, $\eta_{1}$ is an arbitrary fixed function and $\eta_{2}$ is a
fixed function that will be chosen later. Define the real function
\begin{equation*}
\bar{K}(\varepsilon_{1},\varepsilon_{2})=\mathcal{K}(\bar{y})=\left(\int_a^b
K_{\Delta}[\bar{y}](t) \Delta t\right) \left(\int_a^b
K_{\nabla}\{\bar{y}\}(t) \nabla t\right)-k.
\end{equation*}
We have
\begin{multline*}
\left.\frac{\partial\bar{K}}{\partial
\varepsilon_{2}}\right|_{(0,0)} = \mathcal{K}_\nabla(\hat{y})
\int_a^b \left(\partial_2 K_\Delta[\hat{y}](t) \eta_{2}^\sigma(t) +
\partial_3 K_\Delta[\hat{y}](t) \eta_{2}^\Delta(t)\right) \Delta t\\
+ \mathcal{K}_\Delta(\hat{y}) \int_a^b \left(\partial_2
K_\nabla\{\hat{y}\}(t) \eta_{2}^\rho(t) + \partial_3
K_\nabla\{\hat{y}\}(t) \eta_{2}^\nabla(t)\right) \nabla t = 0 \, .
\end{multline*}
We now make use of the following formulas of integration by parts \cite{B:P:01}:
if functions $f,g : \mathbb{T}\rightarrow\mathbb{R}$ are delta and nabla
differentiable with continuous derivatives, then
\begin{equation*}
\begin{split}
\int_{a}^{b}f^\sigma(t) g^{\Delta}(t)\Delta t
&=\left.(fg)(t)\right|_{t=a}^{t=b}
-\int_{a}^{b}f^{\Delta}(t)g(t)\Delta t \, , \\
\int_{a}^{b}f^\rho(t)g^{\nabla}(t)\nabla t
&=\left.(fg)(t)\right|_{t=a}^{t=b}
-\int_{a}^{b}f^{\nabla}(t)g(t)\nabla t \, .
\end{split}
\end{equation*}
Having in mind that $\eta_{2}(a)=\eta_{2}(b)=0$,
we obtain:
\begin{multline*}
\int_a^b \partial_2 K_\Delta[\hat{y}](t) \eta_{2}^\sigma(t)\Delta t=
\int_a^t\partial_2 K_\Delta[\hat{y}](\tau)\Delta \tau
\eta_{2}(t)|^{t=b}_{t=a}\\
-\int_a^b\left(\int_a^t
\partial_2 K_\Delta[\hat{y}](\tau)\Delta\tau\right)\eta_{2}^{\Delta}(t) \Delta
t =-\int_a^b\left(\int_a^t
\partial_2 K_\Delta[\hat{y}](\tau)\Delta\tau\right)\eta_{2}^{\Delta}(t) \Delta t
\end{multline*}
and
\begin{multline*}
\int_a^b \partial_2 K_\nabla\{\hat{y}\}(t) \eta_{2}^\rho(t)\nabla
t=\int_a^t\partial_2 K_\nabla\{\hat{y}\}(\tau)\nabla\tau
\eta_{2}(t)|^{t=b}_{t=a}\\
-\int_a^b\left(\int_a^t\partial_2 K_\nabla\{\hat{y}\}(\tau)\nabla
\tau \right) \eta_{2}^\nabla(t)\nabla t
=-\int_a^b\left(\int_a^t\partial_2 K_\nabla\{\hat{y}\}(\tau)\nabla
\tau \right) \eta_{2}^\nabla(t)\nabla t .
\end{multline*}
Therefore,
\begin{multline}
\label{after:parts}
\left.\frac{\partial\bar{K}}{\partial
\varepsilon_{2}}\right|_{(0,0)}=\mathcal{K}_\nabla(\hat{y}) \int_a^b
\left(\partial_3 K_\Delta[\hat{y}](t)-\int_a^t
\partial_2 K_\Delta[\hat{y}](\tau)\Delta\tau\right)\eta_{2}^{\Delta}(t) \Delta
t\\
+\mathcal{K}_\Delta(\hat{y}) \int_a^b \left(\partial_3
K_\nabla\{\hat{y}\}(t)-\int_a^t\partial_2
K_\nabla\{\hat{y}\}(\tau)\nabla \tau \right)
\eta_{2}^\nabla(t)\nabla t.
\end{multline}
Let $$f(t)=\mathcal{K}_\nabla(\hat{y})\left(\partial_3
K_\Delta[\hat{y}](t)-\int_a^t
\partial_2 K_\Delta[\hat{y}](\tau)\Delta\tau\right)$$
and $$g(t)=\mathcal{K}_\Delta(\hat{y})\left(\partial_3
K_\nabla\{\hat{y}\}(t)-\int_a^t\partial_2
K_\nabla\{\hat{y}\}(\tau)\nabla \tau \right).$$ We can then write
equation \eqref{after:parts} in the form
\begin{equation}\label{after:sub}
\left.\frac{\partial\bar{K}}{\partial
\varepsilon_{2}}\right|_{(0,0)}=\int_a^bf(t)\eta_{2}^{\Delta}(t)
\Delta t+\int_a^bg(t)\eta_{2}^\nabla(t)\nabla t.
\end{equation}
Transforming the delta integral in \eqref{after:sub} to a nabla
integral by means of \eqref{eq:DtoN} we obtain
\begin{equation*}
\left.\frac{\partial\bar{K}}{\partial
\varepsilon_{2}}\right|_{(0,0)}=\int_a^bf^{\rho}(t)(\eta_{2}^{\Delta})^{\rho}(t)
\nabla t+\int_a^bg(t)\eta_{2}^\nabla(t)\nabla t
\end{equation*}
and by \eqref{eq:chgN_to_D}
\begin{equation*}
\left.\frac{\partial\bar{K}}{\partial
\varepsilon_{2}}\right|_{(0,0)}=\int_a^b\left(f^{\rho}(t)
+g(t)\right)\eta_{2}^\nabla(t)\nabla t.
\end{equation*}
As $\hat{y}$ is a normal extremizer we conclude, by
Lemma~\ref{DBRL:n} and equation \eqref{eq:EL2:iso}, that there exists
$\eta_2$ such that $\left.\frac{\partial\bar{K}}{\partial
\varepsilon_{2}}\right|_{(0,0)}\neq 0$. Since
$\bar{K}(0,0)=0$, by the implicit function theorem we conclude that
there exists a function $\varepsilon_{2}$ defined in the
neighborhood of zero, such that
$\bar{K}(\varepsilon_{1},\varepsilon_{2}(\varepsilon_{1}))=0$, \textrm{i.e.},
we may choose a subset of variations $\bar{y}$ satisfying the
isoperimetric constraint.

Let us now consider the real function
\begin{equation*}
\bar{L}(\varepsilon_{1},\varepsilon_{2})=\mathcal{L}(\bar{y})=\left(\int_a^b
L_{\Delta}[\bar{y}](t) \Delta t\right) \left(\int_a^b
L_{\nabla}\{\bar{y}\}(t) \nabla t\right).
\end{equation*}
By hypothesis, $(0,0)$ is an extremal of $\bar{L}$ subject to the
constraint $\bar{K}=0$ and $\nabla \bar{K}(0,0)\neq \textbf{0}$. By
the Lagrange multiplier rule, there exists some real $\lambda$ such
that $\nabla(\bar{L}(0,0)-\lambda\bar{K}(0,0))=\textbf{0}$. Having
in mind that $\eta_{1}(a)=\eta_{1}(b)=0$, we can write
\begin{multline}\label{function:L}
\left.\frac{\partial\bar{L}}{\partial
\varepsilon_{1}}\right|_{(0,0)} =\mathcal{L}_\nabla(\hat{y})
\int_a^b \left(\partial_3 L_\Delta[\hat{y}](t)-\int_a^t
\partial_2 L_\Delta[\hat{y}](\tau)\Delta\tau\right)\eta_{1}^{\Delta}(t)
\Delta t\\
+\mathcal{L}_\Delta(\hat{y}) \int_a^b \left(\partial_3
L_\nabla\{\hat{y}\}(t)-\int_a^t\partial_2
L_\nabla\{\hat{y}\}(\tau)\nabla \tau \right)
\eta_{1}^\nabla(t)\nabla t
\end{multline}
and
\begin{multline}\label{function:K}
\left.\frac{\partial\bar{K}}{\partial
\varepsilon_{1}}\right|_{(0,0)}=\mathcal{K}_\nabla(\hat{y}) \int_a^b
\left(\partial_3 K_\Delta[\hat{y}](t)-\int_a^t
\partial_2 K_\Delta[\hat{y}](\tau)\Delta\tau\right)\eta_{1}^{\Delta}(t) \Delta
t\\
+\mathcal{K}_\Delta(\hat{y}) \int_a^b \left(\partial_3
K_\nabla\{\hat{y}\}(t)-\int_a^t\partial_2
K_\nabla\{\hat{y}\}(\tau)\nabla \tau \right)
\eta_{1}^\nabla(t)\nabla t.
\end{multline}
Let $$m(t)=\mathcal{L}_\nabla(\hat{y})\left(\partial_3
L_\Delta[\hat{y}](t)-\int_a^t
\partial_2 L_\Delta[\hat{y}](\tau)\Delta\tau\right)$$
and $$n(t)=\mathcal{L}_\Delta(\hat{y})\left(\partial_3
L_\nabla\{\hat{y}\}(t)-\int_a^t\partial_2
L_\nabla\{\hat{y}\}(\tau)\nabla \tau \right).$$ Then equations
\eqref{function:L} and \eqref{function:K} can be written in the form
\begin{equation*}
\left.\frac{\partial\bar{L}}{\partial
\varepsilon_{1}}\right|_{(0,0)}=\int_a^bm(t)\eta_{1}^{\Delta}(t)
\Delta t+\int_a^bn(t)\eta_{1}^\nabla(t)\nabla t
\end{equation*}
and
\begin{equation*}
\left.\frac{\partial\bar{K}}{\partial
\varepsilon_{1}}\right|_{(0,0)}=\int_a^bf(t)\eta_{1}^{\Delta}(t)
\Delta t+\int_a^bg(t)\eta_{1}^\nabla(t)\nabla t.
\end{equation*}
Transforming the delta integrals in the above equalities to nabla
integrals by means of \eqref{eq:DtoN} and using \eqref{eq:chgN_to_D}
we obtain
\begin{equation*}
\left.\frac{\partial\bar{L}}{\partial
\varepsilon_{1}}\right|_{(0,0)}
=\int_a^b\left(m^{\rho}(t)+n(t)\right)\eta_{1}^\nabla(t)\nabla t
\end{equation*}
and
\begin{equation*}
\left.\frac{\partial\bar{K}}{\partial
\varepsilon_{1}}\right|_{(0,0)}
=\int_a^b\left(f^{\rho}(t)+g(t)\right)\eta_{1}^\nabla(t)\nabla t.
\end{equation*}
Therefore,
\begin{equation}
\label{iso}
\int_{a}^{b}\eta_{1}^{\Delta}(t)\left\{m^{\rho}(t)+n(t)
-\lambda\left(f^{\rho}(t)+g(t)\right)\right\}\nabla t=0.
\end{equation}
Since \eqref{iso} holds for any $\eta_{1}$, by Lemma~\ref{DBRL:n}
we have
\begin{equation*}
m^{\rho}(t)+n(t)-\lambda\left(f^{\rho}(t)+g(t)\right)=c
\end{equation*}
for some $c\in \mathbb{R}$ and all $t \in [a,b]_\kappa$. Hence,
condition \eqref{iso:EL1} holds. In a similar way we can obtain
equation \eqref{iso:EL2}. In that case we use relationships
\eqref{eq:chgD_to_N} and \eqref{eq:NtoD},
and \cite[Lemma~4.1]{B:CV:2004}.
\end{proof}

In the particular case $L_\nabla \equiv \frac{1}{b-a}$
we get from Theorem~\ref{thm:mr:iso} the main result of
\cite{F:T:09}:

\begin{corollary}[Theorem~3.4 of \cite{F:T:09}]
Suppose that
\begin{equation*}
J(y)=\int_a^b L(t,y^\sigma(t),y^\Delta(t))\Delta t
\end{equation*}
has a local minimum at $y_\ast$
subject to the boundary conditions $y(a)=y_a$ and $y(b)=y_b$
and the isoperimetric constraint
\begin{equation*}
I(y)=\int_a^b g(t,y^\sigma(t),y^\Delta(t))\Delta t = k \, .
\end{equation*}
Assume that $y_\ast$ is not an extremal for the functional $I$.
Then, there exists a Lagrange multiplier constant
$\lambda$ such that $y_\ast$ satisfies the following equation:
\begin{equation*}
\partial_3 F^\Delta(t,y^\sigma_\ast(t),y^\Delta_\ast(t))
-\partial_2 F(t,y^\sigma_\ast(t),y^\Delta_\ast(t))=0 \
\mbox{ for all } \ t\in[a,b]^{\kappa^2},
\end{equation*}
where $F=L-\lambda g$ and $\partial_3 F^\Delta$ denotes the delta
derivative of a composition.
\end{corollary}

One can easily cover abnormal extremizers within our result by
introducing an extra multiplier $\lambda_{0}$.

\begin{theorem}
\label{th:iso:abn}
If $\hat{y} \in C_{\diamond}^1$ is an extremizer for the isoperimetric problem
\eqref{problem:P:iso}--\eqref{const}, then there exist two constants
$\lambda_{0}$ and $\lambda$, not both zero, such that $\hat{y}$
satisfies the following delta-nabla integral equations:
\begin{multline}
\label{iso:EL1:abn}
\lambda_{0}\left\{\mathcal{L}_\nabla(\hat{y})
\left(\partial_3 L_\Delta[\hat{y}](\rho(t))
-\int_{a}^{\rho(t)} \partial_2 L_\Delta[\hat{y}](\tau) \Delta\tau\right)\right.\\
+ \left.\mathcal{L}_\Delta(\hat{y}) \left(\partial_3
L_\nabla\{\hat{y}\}(t) -\int_{a}^{t} \partial_2
L_\nabla\{\hat{y}\}(\tau) \nabla\tau\right) \right\}\\
-\lambda\left\{\mathcal{K}_\nabla(\hat{y}) \left(\partial_3
K_\Delta[\hat{y}](\rho(t))
-\int_{a}^{\rho(t)} \partial_2 K_\Delta[\hat{y}](\tau) \Delta\tau\right)\right.\\
\left.+ \mathcal{K}_\Delta(\hat{y}) \left(\partial_3
K_\nabla\{\hat{y}\}(t) -\int_{a}^{t} \partial_2
K_\nabla\{\hat{y}\}(\tau) \nabla\tau\right)\right\} = \text{const}
\quad \forall t \in [a,b]_\kappa \, ;
\end{multline}
\begin{multline}
\label{iso:EL2:iso}
\lambda_{0}\left\{\mathcal{L}_\nabla(\hat{y})
\left(\partial_3 L_\Delta[\hat{y}](t)
-\int_{a}^{t} \partial_2 L_\Delta[\hat{y}](\tau) \Delta\tau\right)\right.\\
+\left. \mathcal{L}_\Delta(\hat{y}) \left(\partial_3
L_\nabla\{\hat{y}\}(\sigma(t)) -\int_{a}^{\sigma(t)} \partial_2
L_\nabla\{\hat{y}\}(\tau) \nabla\tau\right)\right\} \\
-\lambda \left\{\mathcal{K}_\nabla(\hat{y}) \left(\partial_3
K_\Delta[\hat{y}](t)
-\int_{a}^{t} \partial_2 K_\Delta[\hat{y}](\tau) \Delta\tau\right)\right.\\
+ \left.\mathcal{K}_\Delta(\hat{y}) \left(\partial_3
K_\nabla\{\hat{y}\}(\sigma(t)) -\int_{a}^{\sigma(t)} \partial_2
K_\nabla\{\hat{y}\}(\tau) \nabla\tau\right)\right\}= \text{const}
\quad \forall t \in [a,b]^\kappa \, .
\end{multline}
\end{theorem}

\begin{proof}
Following the proof of Theorem~\ref{thm:mr:iso}, since $(0,0)$ is an
extremal of $\bar{L}$ subject to the constraint $\bar{K}=0$, the
extended Lagrange multiplier rule (see for instance
\cite[Theorem~4.1.3]{Brunt}) asserts the existence of reals
$\lambda_{0}$ and $\lambda$, not both zero, such that
$\nabla(\lambda_{0}\bar{L}(0,0)-\lambda\bar{K}(0,0))=\textbf{0}$.
Therefore,
\begin{equation}\label{iso:abn}
\int_{a}^{b}\eta_{1}^{\Delta}(t)\left\{\lambda_{0}\left(m^{\rho}(t)+n(t)\right)
-\lambda\left(f^{\rho}(t)+g(t)\right)\right\}\nabla t=0.
\end{equation}
Since \eqref{iso:abn} holds for any $\eta_{1}$, by
Lemma~\ref{DBRL:n}, we have
\begin{equation*}
\lambda_{0}\left(m^{\rho}(t)+n(t)\right)-\lambda\left(f^{\rho}(t)+g(t)\right)=c
\end{equation*}
for some $c\in \mathbb{R}$ and all $t \in [a,b]_\kappa$. This
establishes equation \eqref{iso:EL1:abn}. Equation
\eqref{iso:EL2:iso} can be shown using a similar technique.
\end{proof}

\begin{remark}
If $\hat{y} \in C_{\diamond}^1$ is an extremizer for the isoperimetric problem
\eqref{problem:P:iso}--\eqref{const}, then we can choose $\lambda_{0}=1$
in Theorem~\ref{th:iso:abn} and obtain Theorem~\ref{thm:mr:iso}. For
abnormal extremizers, Theorem~\ref{th:iso:abn} holds with
$\lambda_{0}=0$. The condition $(\lambda_{0},\lambda)\neq\textbf{0}$
guarantees that Theorem~\ref{th:iso:abn} is a useful necessary
optimality condition.
\end{remark}

In the particular case $L_\Delta \equiv \frac{1}{b-a}$
we get from Theorem~\ref{th:iso:abn}
the main result of \cite{A:T}:

\begin{corollary}[Theorem~2 of \cite{A:T}]
If $y$ is a local minimizer or maximizer for
\begin{equation*}
I[y]=\int_{a}^{b}f(t,y^\rho(t),y^\nabla(t))\nabla t
\end{equation*}
subject to the boundary conditions $y(a)=\alpha$ and $y(b)=\beta$ and
the nabla-integral constraint
\begin{equation*}
J[y]=\int_{a}^{b}g(t,y^\rho(t),y^\nabla(t))
\nabla t =\Lambda \, ,
\end{equation*}
then there exist two constants $\lambda_0$ and $\lambda$,
not both zero, such that
$$\partial_3 K^\nabla\left(t,y^\rho(t),y^\nabla(t)\right)
-\partial_2 K\left(t,y^\rho(t),y^\nabla(t)\right)=0$$
for all $t \in [a,b]_{\kappa}$,
where $K=\lambda_0 f-\lambda g$.
\end{corollary}

%=============================
\section{An Example}
\label{sec:ex}
%=============================

Let $\mathbb{T}=\{1,2,3,\ldots,M\}$, where $M\in\mathbb{N}$ and
$M\geq 2$. Consider the problem
\begin{equation}\label{ex:iso:1}
\begin{gathered}
\textrm{minimize} \quad
\mathcal{L}(y)=\left(\int_{0}^{M}(y^\Delta(t))^2\Delta
t\right)\left(\int_{0}^{M}\left(y^\nabla(t))^2+y^\nabla(t)\right)\nabla
t\right)\\
y(0)=0, \quad y(M)=M,
\end{gathered}
\end{equation}
subject to the constraint
\begin{equation}\label{ex:iso:2}
\mathcal{K}(y)=\int_{0}^{M}ty^{\Delta}(t)\Delta t=1.
\end{equation}
Since
\begin{equation*}
L_{\Delta}=(y^\Delta)^2,\quad L_{\nabla}=(y^\nabla)^2+y^\nabla,\quad
K_{\Delta}=ty^{\Delta},\quad K_{\nabla}=\frac{1}{M}
\end{equation*}
we have
\begin{equation*}
\partial_2L_{\Delta}=0,\quad \partial_3L_{\Delta}
=2y^\Delta, \quad \partial_2L_{\nabla}=0,
\quad \partial_3L_{\nabla}=2y^\nabla+1,
\end{equation*}
and
\begin{equation*}
\partial_2K_{\Delta}=0,\quad \partial_3K_{\Delta}=t,
\quad \partial_2K_{\nabla}=0,\quad \partial_3K_{\nabla}=0.
\end{equation*}
As
\begin{multline*}
\mathcal{K}_\nabla(\hat{y}) \left(\partial_3K_\Delta[\hat{y}](t)
-\int_{a}^{t} \partial_2 K_\Delta[\hat{y}](\tau) \Delta\tau\right)\\
+ \mathcal{K}_\Delta(\hat{y}) \left(\partial_3
K_\nabla\{\hat{y}\}(\sigma(t)) -\int_{a}^{\sigma(t)} \partial_2
K_\nabla\{\hat{y}\}(\tau) \nabla\tau\right) =t
\end{multline*}
there are no abnormal extremals for the problem
\eqref{ex:iso:1}--\eqref{ex:iso:2}. Applying
equation \eqref{iso:EL2} of Theorem~\ref{thm:mr:iso}
we get the following delta-nabla differential equation:
\begin{equation}\label{ex:iso:3}
2Ay^{\Delta}(t)+B+2By^{\nabla}(\sigma(t))-\lambda t=C,
\end{equation}
where $C\in\mathbb{R}$ and $A$, $B$ are the values of functionals
$\mathcal{L}_{\nabla}$ and $\mathcal{L}_{\Delta}$ in a solution of
\eqref{ex:iso:1}--\eqref{ex:iso:2}, respectively. Since
$y^{\nabla}({\sigma}(t))=y^\Delta(t)$ \eqref{eq:chgD_to_N},
we can write equation \eqref{ex:iso:3} in the form
\begin{equation}
\label{ex:iso:4}
2Ay^{\Delta}(t)+B+2By^\Delta-\lambda t=C.
\end{equation}
Observe that $B\neq 0$ and $A>2$. Hence, solving equation
\eqref{ex:iso:4} subject to the boundary conditions $y(0)=0$ and
$y(M)=M$ we get
\begin{equation}\label{ex:iso:5}
y(t)= \left[1 - \frac{\lambda\left(M-t\right)}{4(A+B)}\right] t \, .
\end{equation}
Substituting \eqref{ex:iso:5} into \eqref{ex:iso:2} we obtain
$\lambda=-\frac{\left( A+B \right) \left( M-2 \right)}{12 M
\left(M-1 \right)}$. Hence,
\begin{equation*}
y(t)=\frac{\left(4\,{M}^{2}-7\,M-3 M \,t + 6\,t\right) t}{M
\left( M-1\right)}
\end{equation*}
is an extremal for the problem \eqref{ex:iso:1}--\eqref{ex:iso:2}.

%=============================
\section{Conclusion}
\label{sec:conc}
%=============================

Minimization of functionals
given by the product of two integrals
were considered by Euler himself,
and are now receiving an increase of interest
due to their nonlocal properties
and applications to economics \cite{Pedregal,Bedlewo:2009}.
In this paper we obtained general necessary optimality conditions
for isoperimetric problems of the calculus of variations
on time scales. Our results extend the ones with delta derivatives
proved in \cite{F:T:09} and analogous nabla
results \cite{A:T} to more general variational problems
described by the product of delta and nabla integrals.

%=============================
\section{Open Problems}
\label{sec:open}
%=============================

The results here obtained can be generalized in different ways:
(i) to variational problems involving higher-order delta and nabla derivatives,
unifying and extending the higher-order results on time scales
of \cite{F:T:08} and \cite{NM:T};
(ii) to problems of the calculus of variations
with a functional which is the composition of a certain scalar
function $H$ with the delta integral of a vector valued field $f_\Delta$
and a nabla integral of a vector field $f_\nabla$, \textrm{i.e.}, of the form
$$H\left(\int_{a}^{b}f_\Delta(t,y^{\sigma}(t),y^{\Delta}(t))\Delta t \, ,
\int_{a}^{b}f_\nabla(t,y^{\rho}(t),y^{\nabla}(t))\nabla t\right)\, .$$
It remains to prove Euler-Lagrange equations and natural boundary conditions
for such problems on time scales, with or without constraints.

Sufficient optimality conditions for delta-nabla problems
of the calculus of variations is a completely open question.
It would be also interesting to study direct optimization
methods, extending the results of \cite{abmalina:delfim}
to the more general delta-nabla setting.\\

% ---------------------------------------------------

{\bf ACKNOWLEDGEMENTS.}
This work was partially supported by the \emph{Portuguese Foundation
for Science and Technology} (FCT) through the
\emph{Systems and Control Group} of the R\&D Unit CIDMA.
The first author is also supported by Bia{\l}ystok University of Technology,
via a project of the Polish Ministry of Science and Higher Education
``Wsparcie miedzynarodowej mobilnosci naukowcow''.

% ---------------------------------------------------

% ---------------------------------------------------


\begin{thebibliography}{99}

\bibitem{A:T:JMAA}
R. Almeida\ and\ D. F. M. Torres,
H\"olderian variational problems subject to integral constraints,
J. Math. Anal. Appl. {\bf 359} (2009), no.~2, 674--681.
{\tt arXiv:0807.3076}

\bibitem{A:T}
R. Almeida\ and\ D. F. M. Torres,
Isoperimetric problems on time scales with nabla derivatives,
J. Vib. Control {\bf 15} (2009), no.~6, 951--958.
{\tt arXiv:0811.3650}

\bibitem{A:G:02}
F. M. Atici\ and\ G. Sh. Guseinov,
On Green's functions and positive solutions for boundary
value problems on time scales,
J. Comput. Appl. Math. {\bf 141} (2002), no.~1-2, 75--99.

\bibitem{Viktor}
V. Bl\aa sj\"o,
The isoperimetric problem,
Amer. Math. Monthly {\bf 112} (2005), no.~6, 526--566.

\bibitem{B:CV:2004}
M. Bohner, Calculus of variations on time scales,
Dynam. Systems Appl. {\bf 13} (2004), no.~3-4, 339--349.

\bibitem{B:P:01}
M. Bohner\ and\ A. Peterson,
{\it Dynamic equations on time scales},
Birkh\"auser Boston, Boston, MA, 2001.

\bibitem{B:P:03}
M. Bohner\ and\ A. Peterson,
{\it Advances in dynamic equations on time scales},
Birkh\"auser Boston, Boston, MA, 2003.

\bibitem{Pedregal}
E. Castillo, A. Luce\~{n}o\ and\ P. Pedregal,
Composition functionals in calculus of variations.
Application to products and quotients,
Math. Models Methods Appl. Sci. {\bf 18} (2008), no.~1, 47--75.

\bibitem{Curtis}
J. P. Curtis,
Complementary extremum principles for isoperimetric optimization problems,
Optim. Eng. {\bf 5} (2004), no.~4, 417--430.

\bibitem{F:T:08}
R. A. C. Ferreira\ and\ D. F. M. Torres,
Higher-order calculus of variations on time scales,
in {\it Mathematical control theory and finance},
149--159, Springer, Berlin, 2008.
{\tt arXiv:0706.3141}

\bibitem{F:T:09}
R. A. Ferreira\ and\ D. F. M. Torres,
Isoperimetric problems of the calculus of variations on time scales,
in {\it Nonlinear Analysis and Optimization II},
Contemporary Mathematics, vol.~514, Amer. Math. Soc.,
Providence, RI, 2010, 123--131.
{\tt arXiv:0805.0278}

\bibitem{G:G:S:05}
M. G\"urses, G. Sh. Guseinov\ and\ B. Silindir,
Integrable equations on time scales,
J. Math. Phys. {\bf 46} (2005), no.~11, 113510, 22 pp.

\bibitem{Lak:book}
V. Lakshmikantham, S. Sivasundaram\ and\ B. Kaymakcalan,
{\it Dynamic systems on measure chains},
Kluwer Acad. Publ., Dordrecht, 1996.

\bibitem{Bedlewo:2009}
A. B. Malinowska\ and\ D. F. M. Torres,
The delta-nabla calculus of variations,
Fasc. Math. {\bf 44} (2010), 75--83.
{\tt arXiv:0912.0494}

\bibitem{abmalina:delfim}
A. B. Malinowska\ and\ D. F. M. Torres,
Leitmann's direct method of optimization
for absolute extrema of certain problems
of the calculus of variations on time scales,
Appl. Math. Comput. {\bf 217} (2010), no.~3, 1158--1162.
{\tt arXiv:1001.1455}

\bibitem{NM:T}
N. Martins\ and\ D. F. M. Torres,
Calculus of variations on time scales with nabla derivatives,
Nonlinear Anal. {\bf 71} (2009), no.~12, e763--e773.
{\tt arXiv:0807.2596}

\bibitem{rachid}
M. R. Sidi Ammi, R. A. C. Ferreira\ and\ D. F. M. Torres,
Diamond-$\alpha$ Jensen's inequality on time scales,
J. Inequal. Appl. {\bf 2008}, Art. ID 576876, 13 pp.
{\tt arXiv:0712.1680}

\bibitem{Brunt}
B. van Brunt,
{\it The calculus of variations},
Springer, New York, 2004.

\end{thebibliography}
\end{document}